\documentclass[11pt,a4paper]{article}
\usepackage[a4paper,margin=2.15cm]{geometry}
\usepackage{amsmath,amssymb,amsthm,mathtools,bm}
\usepackage{booktabs,tabularx,array,multirow}
\usepackage{enumitem}
\usepackage{microtype}
\usepackage{xcolor}
\usepackage[hidelinks]{hyperref}
\usepackage{fancyhdr}
\usepackage{url}

\numberwithin{equation}{section}

\newtheorem{assumption}{Assumption}[section]
\newtheorem{lemma}[assumption]{Lemma}
\newtheorem{proposition}[assumption]{Proposition}
\newtheorem{theorem}[assumption]{Theorem}
\newtheorem{corollary}[assumption]{Corollary}
\theoremstyle{remark}
\newtheorem{remark}[assumption]{Remark}

\newcommand{\R}{\mathbb{R}}
\newcommand{\norm}[1]{\left\lVert #1\right\rVert}

\newcommand{\co}{\operatorname{co}}
\newcommand{\diag}{\operatorname{diag}}

\newcommand{\cL}{\mathcal{L}}
\newcommand{\cD}{\mathcal{D}}

\newcommand{\eps}{\varepsilon}
\newcommand{\dd}{\,\mathrm d}

\title{\bfseries
Preconditioned Three-Term Conjugate Gradient Framework for Nonconvex Finite Minimax Problems}
\author{Wenzhe Zhao\\
\small National Center for Applied Mathematics in Chongqing, Chongqing Normal University}
\date{}

\begin{document}
\maketitle

\begin{abstract}
This paper develops a hyperbolic-majorization preconditioned three-term nonlinear conjugate-gradient framework for nonconvex finite minimax optimization.
An analytic symmetric positive definite metric is derived from a global quadratic majorization of the hyperbolic smoothing model and is used simultaneously as a curvature absorber, a preconditioner, and the line-search energy metric.
With the displacement \(s_{k-1}\) fixed and a variable curvature response \(b_k\), an enhanced three-term direction is introduced together with an adaptive parameter \(\mu_k^\star\) that is maximal under the requirement that the baseline worst-case metric-energy constant be preserved.
The resulting framework yields a Dai--Liao-type conjugacy relation, enhanced sufficient descent, a smoothing-parameter-uniform Armijo lower bound, fixed-smoothing global first-order convergence and complexity, and Clarke-stationary accumulation points under continuation.
\end{abstract}



\section{Introduction}

We consider the finite minimax problem
\begin{equation}
\min_{x\in\R^n} f(x),
\qquad
f(x):=\max_{1\le i\le m} f_i(x),
\label{eq:minimax}
\end{equation}
where the component functions \(f_i\) are continuously differentiable but are not assumed to be convex.
Even when all \(f_i\) are smooth, the maximum in \eqref{eq:minimax} is generally nonsmooth at switching points.
Finite-max and minimax structures occur in robust design, worst-case modeling, multiobjective scalarization, learning, and engineering optimization; recent work has also emphasized smooth reformulations and active-component identification for large finite-max models \cite{RasTamUteda2025}.

A classical strategy for \eqref{eq:minimax} is smoothing.
Exponential smoothing for minimax optimization goes back at least to Xu \cite{Xu2001}, and adaptive smoothing schemes were developed by Polak, Royset and Womersley \cite{PolakRoysetWomersley2003}.
Hyperbolic smoothing was studied systematically by Bagirov, Al Nuaimat and Sultanova \cite{Bagirov2013}, with subsequent comparisons between local and global smoothing \cite{Bagirov2018}.
For conjugate-gradient approaches specifically, Pang, Du and Ju \cite{PangDuJu2016} proposed a smoothing Fletcher--Reeves method, Liu and Zheng \cite{LiuZheng2020} developed a smoothing iterative method based on exponential penalty, and Guo and Wan \cite{GuoWan2022} proposed a smoothing three-term conjugate-gradient method for finite minimax problems.
More recent hyperbolic-smoothing developments include a truncated Newton method for large-component finite minimax problems \cite{ZhaoTang2026} and modified L-BFGS constructions \cite{Zhao2026LBFGS}.
These works motivate the question of how the curvature induced by smoothing itself can be exploited without introducing a full Hessian or a dense quasi-Newton matrix.

Nonlinear conjugate-gradient (NCG) methods are attractive for large-scale smooth optimization because of their low storage requirements.
The classical lineage includes Hestenes--Stiefel \cite{HestenesStiefel1952}, Fletcher--Reeves \cite{FletcherReeves1964}, and globally convergent variants such as Dai--Yuan \cite{DaiYuan1999}.
Global convergence mechanisms for nonlinear CG were studied, among others, by Gilbert and Nocedal \cite{GilbertNocedal1992}.
Dai and Liao \cite{DaiLiao2001} introduced a generalized conjugacy condition compatible with inexact line searches, while Hager and Zhang \cite{HagerZhang2005} designed a nonlinear CG method with guaranteed descent and an efficient line search.

Three-term CG methods provide additional freedom for enforcing sufficient descent and conjugacy.
Zhang, Zhou and Li \cite{ZhangZhouLi2007} developed descent three-term methods for nonconvex unconstrained optimization.
Narushima, Yabe and Ford \cite{NarushimaYabeFord2011} proposed a general three-term form whose search directions satisfy an exact sufficient-descent identity independently of the line search, and they further gave a multistep quasi-Newton instance.
Andrei \cite{Andrei2013Simple} derived a simple three-term method satisfying both descent and conjugacy conditions.
Subsequent work studied spectral-scaling and matrix viewpoints, parameter optimization, and adaptive conjugacy: see, for example, Dong, Liu and He \cite{DongLiuHe2015}, Babaie-Kafaki and Ghanbari \cite{BabaieGhanbari2015}, Dong et al. \cite{DongHanDai2018,DongDaiGhanbari2021}, and Andrei's adaptive matrix-based construction \cite{Andrei2016}.
Recent developments include a sufficiently descending three-term method of Awwal et al. \cite{Awwal2025}, restart-based hybrid three-term methods \cite{WangGao2025}, complexity guarantees for a class of nonconvex three-term methods under Wolfe and Armijo line searches \cite{Wang2026}, a matrix extension of the Dai--Liao parameter \cite{BabaieKafakiMatrixDL2026}, and a general three-term framework for vector optimization \cite{LinZeng2026}.
Thus, three-term CG, sufficient descent, generalized conjugacy, matrix parameterizations, and parameter optimization are established themes in the literature and are not claimed as individually new here.

The present work focuses on a different coupling.
For the hyperbolic smoothing of \eqref{eq:minimax}, the scalar smoothing curvature is \(O(\tau^{-1})\) near switching surfaces when the smoothing parameter \(\tau\downarrow0\).
A Euclidean Lipschitz-gradient estimate may therefore lead to line-search constants that deteriorate with \(\tau\).
Instead, we derive directly from the hyperbolic model a structured metric
\[
P_k
=
\delta I+\bar L\Pi_x+\frac12J_k^T D_k^{\rm hyp}J_k
\]
that globally majorizes the whole smoothed objective at the current point.
The same \(P_k\) is then used in the three-term direction.
The main direction keeps the classical CG displacement \(s_{k-1}\) fixed, while allowing a curvature-response vector \(b_k\) to vary.
A nonnegative enhancement parameter \(\mu_k\) is added to the exact-cancellation three-term structure, and an adaptive value \(\mu_k^\star\) is obtained from a precise metric-energy budget.

The main contributions are summarized as follows.
\begin{enumerate}[label=(\roman*),leftmargin=8mm]
\item
For the hyperbolic smoothing model, we derive an analytic SPD matrix \(P_k\) and prove a global quadratic upper model valid for every trial displacement.
Under \(C^2\) component functions, \(P_k\) also strictly dominates the current smoothing Hessian in the Loewner order.
\item
We introduce a fixed-displacement/variable-curvature three-term NCG family
\[
d_k=-P_k^{-1}g_k+\beta_kd_{k-1}-\gamma_kP_k^{-1}b_k,
\]
where \(b_k\) may represent a secant, Hessian-vector, quasi-Newton, finite-difference, or historical curvature response.
The direction satisfies an enhanced sufficient-descent identity; with the regularized curvature vector it also satisfies a generalized conjugacy relation, which becomes a modified Dai--Liao relation for the standard choice \(b_k=y_{k-1}\).
\item
We derive an adaptive parameter \(\mu_k^\star\) as the largest nonnegative enhancement for which the original worst-case \(P_k\)-energy constant is unchanged.
Hence the enhanced branch can produce strictly more descent than the exact-cancellation branch without weakening the baseline theoretical Armijo constant.
\item
Combining the direction estimate with the global majorizer yields an Armijo backtracking lower bound independent of \(\tau\).
For each fixed \(\tau>0\), compactness of the initial sublevel set implies
\(\|\nabla\Phi_\tau(z_k)\|\to0\), together with an \(O(\eps^{-2})\) first-order iteration bound.
\item
For a continuation sequence \(\tau_j\downarrow0\) with vanishing inner stationarity tolerances, every accumulation point is Clarke stationary for the original finite minimax problem.
We also record a broader \((a_k,b_k)\) extension and several admissible variations as remarks, while keeping the main algorithm within the standard three-term NCG recurrence.
\end{enumerate}

The paper is organized as follows.
Section~\ref{sec:prelim} gives the minimax reformulation and hyperbolic smoothing model.
Section~\ref{sec:majorizer} derives the majorization metric.
Section~\ref{sec:direction} develops the variable-curvature three-term direction and its conjugacy and descent properties.
Section~\ref{sec:mu} derives the energy-budget parameter.
Sections~\ref{sec:algorithm}--\ref{sec:continuation} establish the Armijo, convergence, complexity, and continuation results.
Section~\ref{sec:instances} gives concrete curvature-response instances, Section~\ref{sec:implementation} discusses implementation, and Section~\ref{sec:numerics} is reserved for numerical experiments.

\section{Finite minimax reformulation and hyperbolic smoothing}
\label{sec:prelim}

We use the Euclidean inner product and norm.
For a symmetric matrix \(M\), \(M\succ0\) and \(M\succeq0\) denote positive definiteness and semidefiniteness.
The active index set of \eqref{eq:minimax} is
\[
I(x):=\{i:f_i(x)=f(x)\}.
\]
Since the maximum of finitely many \(C^1\) functions is locally Lipschitz,
\begin{equation}
\partial_C f(x)
=
\co\{\nabla f_i(x):i\in I(x)\}.
\label{eq:clarke}
\end{equation}

\begin{assumption}[Component smoothness]
\label{ass:smooth}
For every \(i\in\{1,\ldots,m\}\), \(f_i\in C^1(\R^n)\), and there exists \(L_i\ge0\) such that
\[
\norm{\nabla f_i(x)-\nabla f_i(y)}
\le
L_i\norm{x-y},
\qquad x,y\in\R^n.
\]
Set
\[
\bar L:=\sum_{i=1}^m L_i.
\]
\end{assumption}

\begin{remark}
The global Lipschitz statement is used for a clean global majorization theorem.
The same analysis only needs valid Lipschitz constants on an open convex set containing all accepted iterates and all line-search trial segments.
\end{remark}

Introduce \(z=(x,t)\in\R^{n+1}\) and
\[
\Pi_x:=\diag(I_n,0).
\]
The positive-part reformulation is
\begin{equation}
F(x,t)
=
t+\sum_{i=1}^m[f_i(x)-t]_+.
\label{eq:F}
\end{equation}

\begin{proposition}[Exact auxiliary-variable reformulation]
\label{prop:reform}
For every \(x\in\R^n\),
\[
f(x)=\min_{t\in\R}F(x,t),
\qquad
F(x,f(x))=f(x).
\]
\end{proposition}

\begin{proof}
Let \(M=f(x)\).
If \(t\ge M\), then \(F(x,t)=t\ge M\).
If \(t<M\), choose \(i^\star\) with \(f_{i^\star}(x)=M\); then
\[
F(x,t)\ge t+[M-t]_+=M.
\]
Equality holds at \(t=M\).
\end{proof}

For \(\tau>0\), define the hyperbolic smoothing of the positive part by
\begin{equation}
\phi_\tau(r)
=
\frac{r+\sqrt{r^2+\tau^2}}2
\label{eq:phi}
\end{equation}
and the smooth objective
\begin{equation}
\Phi_\tau(x,t)
=
t+\sum_{i=1}^m\phi_\tau(f_i(x)-t).
\label{eq:Phi}
\end{equation}

\begin{lemma}[Approximation of the positive part]
\label{lem:approx}
For all \(r\in\R\),
\[
[r]_+
\le
\phi_\tau(r)
\le
[r]_++\frac{\tau}{2}.
\]
Consequently,
\[
F(x,t)
\le
\Phi_\tau(x,t)
\le
F(x,t)+\frac{m\tau}{2},
\]
and
\[
f(x)
\le
\inf_t\Phi_\tau(x,t)
\le
f(x)+\frac{m\tau}{2}.
\]
\end{lemma}

\begin{proof}
Since \(\sqrt{r^2+\tau^2}\ge |r|\),
\(\phi_\tau(r)\ge(r+|r|)/2=[r]_+\).
Also,
\(\sqrt{r^2+\tau^2}\le |r|+\tau\), which gives the upper bound.
The remaining statements follow by summation and Proposition~\ref{prop:reform}.
\end{proof}

Define the residuals and hyperbolic weights
\begin{equation}
r_i(z)=f_i(x)-t,
\qquad
\omega_i(z,\tau)=\sqrt{r_i(z)^2+\tau^2},
\label{eq:res}
\end{equation}
\begin{equation}
a_i(z,\tau)
=
\phi_\tau'(r_i(z))
=
\frac12\left(1+\frac{r_i(z)}{\omega_i(z,\tau)}\right),
\qquad
c_i(z,\tau)
=
\phi_\tau''(r_i(z))
=
\frac{\tau^2}{2\omega_i(z,\tau)^3}.
\label{eq:weights}
\end{equation}
Then \(0<a_i<1\).
Let
\[
q_i(x)=
\begin{pmatrix}
\nabla f_i(x)\\ -1
\end{pmatrix},
\qquad
J(z)=
\begin{pmatrix}
q_1(x)^T\\ \vdots\\ q_m(x)^T
\end{pmatrix}.
\]

\begin{proposition}[Gradient and Hessian structure]
\label{prop:grad-hess}
Under Assumption~\ref{ass:smooth},
\begin{equation}
\nabla\Phi_\tau(z)
=
\begin{pmatrix}
\displaystyle\sum_{i=1}^m a_i(z,\tau)\nabla f_i(x)\\[1mm]
\displaystyle 1-\sum_{i=1}^m a_i(z,\tau)
\end{pmatrix}.
\label{eq:gradient}
\end{equation}
If, in addition, all \(f_i\in C^2\), then
\begin{equation}
\nabla^2\Phi_\tau(z)
=
\begin{pmatrix}
\displaystyle\sum_{i=1}^m a_i(z,\tau)\nabla^2f_i(x)&0\\
0&0
\end{pmatrix}
+
J(z)^TC(z,\tau)J(z),
\label{eq:hessian}
\end{equation}
where
\[
C(z,\tau)=\diag(c_1(z,\tau),\ldots,c_m(z,\tau)).
\]
\end{proposition}

\begin{proof}
The gradient formula follows from the chain rule.
Differentiating once more yields
\[
\nabla^2_{xx}\Phi_\tau
=
\sum_i a_i\nabla^2f_i
+
\sum_i c_i\nabla f_i\nabla f_i^T,
\quad
\nabla^2_{xt}\Phi_\tau=-\sum_i c_i\nabla f_i,
\quad
\partial^2_{tt}\Phi_\tau=\sum_i c_i,
\]
which is exactly \eqref{eq:hessian}.
\end{proof}

\begin{remark}
The scalar hyperbolic curvature satisfies
\[
\max_r \phi_\tau''(r)=\frac1{2\tau}.
\]
Thus Euclidean curvature estimates can grow like \(O(\tau^{-1})\) near switching surfaces.
The metric constructed next is designed to absorb this curvature directly.
\end{remark}

\section{The hyperbolic global-majorization metric}
\label{sec:majorizer}

Fix an inner smoothing parameter \(\tau>0\) and an iterate
\(z_k=(x_k,t_k)\).
Write
\[
r_{i,k}=f_i(x_k)-t_k,
\qquad
\omega_{i,k}=\sqrt{r_{i,k}^2+\tau^2},
\]
and let \(J_k=J(z_k)\).
Define
\begin{equation}
D_k^{\rm hyp}
=
\diag(\omega_{1,k}^{-1},\ldots,\omega_{m,k}^{-1}).
\label{eq:Dhyp}
\end{equation}
For a fixed regularization \(\delta>0\), set
\begin{equation}
\boxed{
P_k
=
\delta I_{n+1}
+
\bar L\Pi_x
+
\frac12J_k^TD_k^{\rm hyp}J_k.
}
\label{eq:Pk}
\end{equation}

\begin{proposition}[Positive definiteness]
\label{prop:spd}
For every \(k\) and every \(\tau>0\),
\[
P_k\succeq\delta I_{n+1}\succ0.
\]
\end{proposition}

\begin{proof}
Both \(\bar L\Pi_x\) and \(J_k^TD_k^{\rm hyp}J_k\) are positive semidefinite.
\end{proof}

\begin{theorem}[Pointwise Loewner domination]
\label{thm:loewner}
Suppose Assumption~\ref{ass:smooth} holds and \(f_i\in C^2\).
Then
\[
P_k-\nabla^2\Phi_\tau(z_k)
\succeq
\delta I_{n+1}.
\]
\end{theorem}

\begin{proof}
Since \(\nabla f_i\) is \(L_i\)-Lipschitz and \(f_i\in C^2\),
\(\nabla^2f_i(x_k)\preceq L_iI\).
Because \(0<a_{i,k}<1\),
\[
\sum_i a_{i,k}\nabla^2f_i(x_k)
\preceq
\sum_i a_{i,k}L_iI
\preceq
\bar LI.
\]
Moreover,
\[
\frac1{2\omega_{i,k}}
-
\frac{\tau^2}{2\omega_{i,k}^3}
=
\frac{r_{i,k}^2}{2\omega_{i,k}^3}
\ge0.
\]
Combining this inequality with Proposition~\ref{prop:grad-hess} gives the result.
\end{proof}

The pointwise Hessian comparison is informative, but line search requires a function-value estimate valid for an arbitrary trial step.
The next scalar inequality is the key.

\begin{lemma}[Hyperbolic scalar majorization]
\label{lem:scalar-major}
For arbitrary \(u,v\in\R\), let
\(\omega=\sqrt{u^2+\tau^2}\).
Then
\begin{equation}
\boxed{
\phi_\tau(u+v)
\le
\phi_\tau(u)+\phi_\tau'(u)v+\frac{v^2}{4\omega}.
}
\label{eq:scalar-major}
\end{equation}
\end{lemma}

\begin{proof}
The square-root function is concave on \((0,\infty)\), hence
\[
\sqrt{(u+v)^2+\tau^2}
\le
\sqrt{u^2+\tau^2}
+
\frac{(u+v)^2-u^2}{2\omega}.
\]
Using
\((u+v)^2-u^2=2uv+v^2\)
and substituting in \eqref{eq:phi} yields \eqref{eq:scalar-major}.
\end{proof}

\begin{lemma}[Monotonicity and Lipschitz property]
\label{lem:phi-lip}
For every \(\tau>0\), \(\phi_\tau\) is strictly increasing and \(1\)-Lipschitz:
\[
|\phi_\tau(u)-\phi_\tau(v)|\le|u-v|.
\]
\end{lemma}

\begin{proof}
Since \(0<\phi_\tau'(r)<1\), the conclusion follows from the mean-value theorem.
\end{proof}

\begin{theorem}[Global quadratic upper model]
\label{thm:global-major}
Let Assumption~\ref{ass:smooth} hold and
\(g_k=\nabla\Phi_\tau(z_k)\).
Then, for every \(d=(p,\eta)\in\R^{n+1}\),
\begin{equation}
\boxed{
\Phi_\tau(z_k+d)
\le
\Phi_\tau(z_k)+g_k^Td+\frac12d^TP_kd.
}
\label{eq:global-major}
\end{equation}
\end{theorem}

\begin{proof}
The descent lemma for each possibly nonconvex \(f_i\) gives
\[
f_i(x_k+p)
\le
f_i(x_k)+\nabla f_i(x_k)^Tp+\frac{L_i}{2}\norm{p}^2.
\]
Define
\[
u_{i,k}(d)=\nabla f_i(x_k)^Tp-\eta=q_i(x_k)^Td.
\]
Then
\[
f_i(x_k+p)-(t_k+\eta)
\le
r_{i,k}+u_{i,k}(d)+\frac{L_i}{2}\norm{p}^2.
\]
By Lemma~\ref{lem:phi-lip},
\[
\phi_\tau(f_i(x_k+p)-t_k-\eta)
\le
\phi_\tau(r_{i,k}+u_{i,k}(d))
+\frac{L_i}{2}\norm{p}^2.
\]
Applying Lemma~\ref{lem:scalar-major},
\[
\phi_\tau(r_{i,k}+u_{i,k}(d))
\le
\phi_\tau(r_{i,k})
+
a_{i,k}u_{i,k}(d)
+
\frac{u_{i,k}(d)^2}{4\omega_{i,k}}.
\]
Summing over \(i\) and adding \(t_k+\eta\), the linear terms form \(g_k^Td\), and
\[
\Phi_\tau(z_k+d)
\le
\Phi_\tau(z_k)
+
g_k^Td
+
\frac{\bar L}{2}\norm{p}^2
+
\frac14\sum_i\frac{(q_i(x_k)^Td)^2}{\omega_{i,k}}.
\]
The last two terms equal
\[
\frac12d^T\left(
\bar L\Pi_x+\frac12J_k^TD_k^{\rm hyp}J_k
\right)d.
\]
Adding the nonnegative term \(\delta\norm d^2/2\) proves \eqref{eq:global-major}.
\end{proof}

\begin{corollary}
\label{cor:ray-major}
For every direction \(d_k\) and every \(\alpha\ge0\),
\[
\Phi_\tau(z_k+\alpha d_k)
\le
\Phi_\tau(z_k)
+
\alpha g_k^Td_k
+
\frac{\alpha^2}{2}d_k^TP_kd_k.
\]
\end{corollary}

\begin{remark}[Role of \(P_k\)]
The algebraic three-term identities developed below only need \(P_k\) to be SPD.
The special hyperbolic form \eqref{eq:Pk} is essential for globalization: the same matrix that defines the search geometry also globally majorizes the smoothing objective.
This identity of the preconditioning metric and the majorization metric is what allows the direction-energy estimate to be converted directly into a \(\tau\)-uniform Armijo bound.
\end{remark}

\section{Variable-curvature three-term CG framework}
\label{sec:direction}

For \(k\ge1\), define the displacement
\[
s_{k-1}=z_k-z_{k-1}.
\]
The main framework fixes this displacement and allows a nonzero curvature-response vector
\[
b_k\in\R^{n+1},\qquad b_k\neq0
\]
to vary with the iteration.
The standard secant choice is \(b_k=y_{k-1}:=g_k-g_{k-1}\), but other choices will be given in Section~\ref{sec:instances}.

For a nonrestart iteration, abbreviate
\[
s=s_{k-1},\quad b=b_k,\quad P=P_k,\quad g=g_k,
\]
and define
\begin{equation}
C=s^TPs,
\qquad
B=b^TP^{-1}b,
\qquad
\Delta=s^Tb,
\qquad
\chi=\frac{\Delta}{\sqrt{CB}}.
\label{eq:C-B-chi}
\end{equation}
By metric Cauchy--Schwarz, \(\chi\in[-1,1]\).
Fix \(\vartheta>1\) and define
\begin{equation}
\boxed{
\cD
=
\Delta+\vartheta\sqrt{CB}
=
\sqrt{CB}(\vartheta+\chi).
}
\label{eq:denom}
\end{equation}

\begin{lemma}[Strict denominator positivity]
\label{lem:den-pos}
If \(s\neq0\) and \(b\neq0\), then
\[
\cD
\ge
(\vartheta-1)\sqrt{CB}>0.
\]
\end{lemma}

\begin{proof}
Since
\[
|\Delta|
=
|(P^{1/2}s)^T(P^{-1/2}b)|
\le
\sqrt{CB},
\]
the result follows from \eqref{eq:denom}.
\end{proof}

For \(\mu_k\ge0\), define the enhanced direction
\begin{equation}
\boxed{
\begin{aligned}
d_k={}&-P_k^{-1}g_k\\
&+
\left[
\frac{g_k^TP_k^{-1}b_k}{\cD_k}
-
\mu_k
\frac{g_k^Ts_{k-1}}{C_k}
\right]s_{k-1}\\
&-
\frac{g_k^Ts_{k-1}}{\cD_k}P_k^{-1}b_k.
\end{aligned}}
\label{eq:direction}
\end{equation}
If \(k=0\), \(s_{k-1}=0\), or \(b_k=0\), we use the restart
\begin{equation}
d_k=-P_k^{-1}g_k.
\label{eq:restart}
\end{equation}

\begin{proposition}[Three-term NCG recurrence]
\label{prop:three-term}
Since \(s_{k-1}=\alpha_{k-1}d_{k-1}\), direction \eqref{eq:direction} can be written as
\begin{equation}
\boxed{
d_k
=
-P_k^{-1}g_k
+
\beta_kd_{k-1}
-
\gamma_kP_k^{-1}b_k,
}
\label{eq:three-term}
\end{equation}
where
\[
\beta_k
=
\alpha_{k-1}
\left[
\frac{g_k^TP_k^{-1}b_k}{\cD_k}
-
\mu_k\frac{g_k^Ts_{k-1}}{C_k}
\right],
\qquad
\gamma_k
=
\frac{g_k^Ts_{k-1}}{\cD_k}.
\]
Thus the main framework remains a preconditioned three-term nonlinear conjugate-gradient method.
\end{proposition}

\begin{theorem}[Enhanced sufficient-descent identity]
\label{thm:descent}
For every nonrestart step,
\begin{equation}
\boxed{
g_k^Td_k
=
-g_k^TP_k^{-1}g_k
-
\mu_k
\frac{(g_k^Ts_{k-1})^2}{C_k}.
}
\label{eq:descent-identity}
\end{equation}
Hence, for every \(\mu_k\ge0\),
\begin{equation}
\boxed{
-g_k^Td_k
\ge
g_k^TP_k^{-1}g_k>0.
}
\label{eq:descent}
\end{equation}
If \(\mu_k>0\) and \(g_k^Ts_{k-1}\neq0\), the inequality is strict.
\end{theorem}

\begin{proof}
Premultiplying \eqref{eq:direction} by \(g_k^T\), the two mixed curvature terms cancel:
\[
\frac{(g_k^TP_k^{-1}b_k)(g_k^Ts_{k-1})}{\cD_k}
-
\frac{(g_k^Ts_{k-1})(g_k^TP_k^{-1}b_k)}{\cD_k}
=0.
\]
The remaining terms give \eqref{eq:descent-identity}.
\end{proof}

Define the regularized curvature response
\begin{equation}
\boxed{
\widetilde b_k
=
b_k
+
\vartheta\sqrt{\frac{B_k}{C_k}}\,P_ks_{k-1}.
}
\label{eq:btilde}
\end{equation}
Then
\begin{equation}
s_{k-1}^T\widetilde b_k
=
\cD_k>0.
\label{eq:secant-pos}
\end{equation}

\begin{theorem}[Regularized curvature-pair conjugacy]
\label{thm:conjugacy}
Direction \eqref{eq:direction} satisfies
\begin{equation}
\boxed{
d_k^T\widetilde b_k
=
-\theta_k\,g_k^Ts_{k-1},
}
\label{eq:conjugacy}
\end{equation}
where
\begin{equation}
\boxed{
\theta_k
=
\frac{
\widetilde b_k^TP_k^{-1}\widetilde b_k
}{
s_{k-1}^T\widetilde b_k
}
+
\mu_k
\frac{
s_{k-1}^T\widetilde b_k
}{
s_{k-1}^TP_ks_{k-1}
}
>0.
}
\label{eq:theta-k}
\end{equation}
\end{theorem}

\begin{proof}
Let \(d_k^{(0)}\) denote \eqref{eq:direction} with \(\mu_k=0\).
Using \eqref{eq:btilde} and \eqref{eq:secant-pos}, direct cancellation gives
\[
d_k^{(0)}
=
-P_k^{-1}g_k
+
\frac{(P_k^{-1}g_k)^T\widetilde b_k}
{s_{k-1}^T\widetilde b_k}s_{k-1}
-
\frac{g_k^Ts_{k-1}}
{s_{k-1}^T\widetilde b_k}
P_k^{-1}\widetilde b_k.
\]
Therefore
\[
(d_k^{(0)})^T\widetilde b_k
=
-
\frac{
\widetilde b_k^TP_k^{-1}\widetilde b_k
}{
s_{k-1}^T\widetilde b_k
}
g_k^Ts_{k-1}.
\]
Since
\[
d_k
=
d_k^{(0)}
-
\mu_k
\frac{g_k^Ts_{k-1}}{C_k}s_{k-1},
\]
equations \eqref{eq:conjugacy}--\eqref{eq:theta-k} follow.
\end{proof}

\begin{corollary}[Modified Dai--Liao specialization]
\label{cor:DL}
If
\[
b_k=y_{k-1}=g_k-g_{k-1},
\]
then \eqref{eq:conjugacy} is a modified Dai--Liao generalized conjugacy relation with regularized secant vector
\[
\widetilde y_{k-1}
=
y_{k-1}
+
\vartheta
\sqrt{
\frac{y_{k-1}^TP_k^{-1}y_{k-1}}
{s_{k-1}^TP_ks_{k-1}}
}\,P_ks_{k-1}.
\]
If the preceding step is an exact line search, then
\(g_k^Ts_{k-1}=0\), and hence
\[
d_k^T\widetilde y_{k-1}=0.
\]
\end{corollary}

\begin{remark}
For a general \(b_k\), relation \eqref{eq:conjugacy} is best interpreted as a Dai--Liao-type curvature-pair conjugacy.
The terminology ``modified Dai--Liao'' is used without qualification for the standard secant response \(b_k=y_{k-1}\).
\end{remark}

\begin{remark}[Positive scaling invariance]
For any \(c_k>0\), replacing \(b_k\) by \(\widehat b_k=c_kb_k\) does not change the direction.
Indeed,
\[
\widehat B_k=c_k^2B_k,\quad
\widehat\Delta_k=c_k\Delta_k,\quad
\widehat\cD_k=c_k\cD_k,
\quad
\widehat\chi_k=\chi_k,
\]
and both curvature terms in \eqref{eq:direction} are unchanged.
Thus only the metric direction of the curvature response, not its arbitrary positive scale, is relevant.
\end{remark}

\begin{remark}[Recovery of the baseline HSPTCG direction]
With the standard secant response \(b_k=y_{k-1}\) and \(\mu_k=0\),
direction \eqref{eq:direction} reduces to the baseline regularized-secant
hyperbolic-smoothing preconditioned three-term CG direction.
The choice \(\mu_k=\mu_k^\star\) therefore defines an enhanced-descent branch
built on the same secant geometry rather than an unrelated search direction.
\end{remark}

\begin{proposition}[Relation with the Narushima exact-cancellation structure]
\label{prop:narushima-relation}
Suppose \(\mu_k=0\) and set \(p_k=P_k^{-1}b_k\).
If \(g_k^Tp_k\neq0\), then \eqref{eq:direction} can be written as
\[
d_k
=
-P_k^{-1}g_k
+
\widehat\beta_kd_{k-1}
-
\widehat\beta_k
\frac{g_k^Td_{k-1}}{g_k^Tp_k}\,p_k,
\]
where
\[
\widehat\beta_k
=
\alpha_{k-1}\frac{g_k^Tp_k}{\cD_k}.
\]
Hence the \(\mu_k=0\) branch is a particular \(P_k\)-metric analogue of the Narushima--Yabe--Ford exact-cancellation construction.
Conversely, if \(\mu_k>0\) and \(g_k^Ts_{k-1}\neq0\), then
\[
g_k^Td_k
<
-g_k^TP_k^{-1}g_k,
\]
so the direction is not in the metric exact-cancellation subclass.
\end{proposition}

\begin{proof}
The first statement follows from
\(s_{k-1}=\alpha_{k-1}d_{k-1}\)
and a direct rearrangement of the \(\mu_k=0\) formula.
The second follows from Theorem~\ref{thm:descent}.
\end{proof}

\begin{remark}
Proposition~\ref{prop:narushima-relation} does not assert an overall containment relation with the 2011 framework.
Narushima--Yabe--Ford allow arbitrary admissible choices of their scalar parameter and third vector, whereas the present exact-cancellation branch imposes the regularized metric denominator \(\cD_k\).
At the same time, the enhanced branch \(\mu_k>0\) leaves the exact-cancellation class.
\end{remark}

\section{Adaptive energy-budget enhancement}
\label{sec:mu}

We now choose \(\mu_k\) so that additional descent is obtained without worsening the baseline worst-case direction-energy constant.

Define
\begin{equation}
\zeta_\vartheta
=
1-\vartheta^{-2},
\qquad
\Gamma_0
=
\zeta_\vartheta^{-1}
=
\frac{\vartheta^2}{\vartheta^2-1}.
\label{eq:gamma0}
\end{equation}
For one iteration, set
\[
h=P^{-1/2}g,
\qquad
u=\frac{P^{1/2}s}{\sqrt C},
\qquad
v=\frac{P^{-1/2}b}{\sqrt B}.
\]
Then
\[
\norm u=\norm v=1,\qquad u^Tv=\chi.
\]
From \eqref{eq:direction},
\begin{equation}
P^{1/2}d
=
-M(\mu)h,
\qquad
M(\mu)
=
I+\mu uu^T
-
\frac{uv^T-vu^T}{\vartheta+\chi}.
\label{eq:Mmu}
\end{equation}
Define
\begin{equation}
\boxed{
\varpi^2
=
\frac{1-\chi^2}{(\vartheta+\chi)^2}.
}
\label{eq:varpi}
\end{equation}

\begin{lemma}[Worst-case skew energy]
\label{lem:varpi}
For \(\chi\in[-1,1]\),
\[
0\le\varpi^2
\le
\frac1{\vartheta^2-1}
=
\Gamma_0-1.
\]
The maximum is attained at \(\chi=-1/\vartheta\).
\end{lemma}

\begin{proof}
For
\[
\psi(\chi)=\frac{1-\chi^2}{(\vartheta+\chi)^2},
\]
we have
\[
\psi'(\chi)
=
-\frac{2(1+\vartheta\chi)}{(\vartheta+\chi)^3}.
\]
Because \(\vartheta+\chi>0\) on \([-1,1]\), the unique interior maximizer is
\(\chi=-1/\vartheta\), where
\(\psi=1/(\vartheta^2-1)\).
\end{proof}

\begin{theorem}[Maximal energy-preserving enhancement]
\label{thm:mu-star}
Let
\[
R_k=\Gamma_0-1-\varpi_k^2\ge0.
\]
The largest \(\mu\ge0\) satisfying
\begin{equation}
M(\mu)^TM(\mu)\preceq\Gamma_0I
\label{eq:budget-matrix}
\end{equation}
is
\begin{equation}
\boxed{
\mu_k^\star
=
\frac{\Gamma_0-1-\varpi_k^2}
{\sqrt{\Gamma_0}+1}.
}
\label{eq:mu-star}
\end{equation}
Moreover,
\[
0\le\mu_k^\star\le\sqrt{\Gamma_0}-1,
\]
and this upper bound is independent of \(\tau\).
\end{theorem}

\begin{proof}
If \(|\chi|<1\), choose an orthonormal basis of \(\operatorname{span}\{u,v\}\) whose first vector is \(u\).
In this basis, the nontrivial \(2\times2\) block of \(M(\mu)\) is
\[
\widehat M(\mu)
=
\begin{pmatrix}
1+\mu&-\varpi\\
\varpi&1
\end{pmatrix}.
\]
(The cases \(|\chi|=1\) follow directly or by continuity, with \(\varpi=0\).)
Thus
\[
\widehat M(\mu)^T\widehat M(\mu)
=
\begin{pmatrix}
(1+\mu)^2+\varpi^2&-\mu\varpi\\
-\mu\varpi&1+\varpi^2
\end{pmatrix}.
\]
Let \(R=\Gamma_0-1-\varpi^2\).
The condition
\(\Gamma_0I-\widehat M(\mu)^T\widehat M(\mu)\succeq0\)
is, for \(\mu\ge0\), equivalent to
\[
(\Gamma_0-1)\mu^2+2R\mu-R^2\le0.
\]
Its positive root is
\[
\mu
=
\frac{R(\sqrt{\Gamma_0}-1)}{\Gamma_0-1}
=
\frac{R}{\sqrt{\Gamma_0}+1},
\]
which proves \eqref{eq:mu-star}.
Finally \(R\le\Gamma_0-1\) gives
\[
\mu_k^\star
\le
\frac{\Gamma_0-1}{\sqrt{\Gamma_0}+1}
=
\sqrt{\Gamma_0}-1.
\]
\end{proof}

\begin{corollary}[Admissible partial enhancement]
\label{cor:partial-mu}
Every choice
\[
0\le\mu_k\le\mu_k^\star
\]
satisfies
\begin{equation}
\boxed{
d_k^TP_kd_k
\le
\Gamma_0\,g_k^TP_k^{-1}g_k.
}
\label{eq:energy}
\end{equation}
In particular one may take
\[
\mu_k=\xi_k\mu_k^\star,
\qquad
0\le\xi_k\le1.
\]
\end{corollary}

\begin{proof}
The feasible set in the scalar quadratic inequality used in the proof of Theorem~\ref{thm:mu-star} is the interval \([0,\mu_k^\star]\).
Then \eqref{eq:Mmu} gives
\[
d_k^TP_kd_k
=
\norm{M(\mu_k)h}^2
\le
\Gamma_0\norm h^2
=
\Gamma_0\,g_k^TP_k^{-1}g_k.
\]
\end{proof}

\begin{theorem}[Uniform descent-to-energy estimate]
\label{thm:energy-descent}
For every \(0\le\mu_k\le\mu_k^\star\),
\begin{equation}
\boxed{
-g_k^Td_k
\ge
\zeta_\vartheta\,d_k^TP_kd_k,
\qquad
\zeta_\vartheta=1-\vartheta^{-2}.
}
\label{eq:eta}
\end{equation}
The constant is identical to the baseline \(\mu_k=0\) worst-case constant and is independent of \(b_k\), \(k\), and \(\tau\).
\end{theorem}

\begin{proof}
Theorem~\ref{thm:descent} gives
\[
-g_k^Td_k
\ge
g_k^TP_k^{-1}g_k.
\]
Corollary~\ref{cor:partial-mu} gives
\[
d_k^TP_kd_k
\le
\Gamma_0 g_k^TP_k^{-1}g_k.
\]
Since \(\Gamma_0^{-1}=\zeta_\vartheta\), the result follows.
\end{proof}

\begin{remark}[Interpretation of \(\mu_k^\star\)]
The quantity \(\mu_k^\star\) uses the local metric-energy margin left by the regularized curvature geometry.
At the worst metric angle \(\chi_k=-1/\vartheta\), one has
\(\varpi_k^2=\Gamma_0-1\) and \(\mu_k^\star=0\), so the method automatically returns to the exact-cancellation branch.
For better-aligned curvature responses, the skew part consumes less of the budget and a larger symmetric descent enhancement is permitted.
\end{remark}

\begin{remark}[Iteration-dependent \(\vartheta_k\)]
The constant \(\vartheta\) may be replaced by \(\vartheta_k\) provided
\[
\vartheta_k\ge\underline\vartheta>1.
\]
Using the local
\[
\Gamma_k=\frac{\vartheta_k^2}{\vartheta_k^2-1}
\]
in \eqref{eq:mu-star} yields
\[
d_k^TP_kd_k
\le
\Gamma_k g_k^TP_k^{-1}g_k
\le
\frac{\underline\vartheta^2}
{\underline\vartheta^2-1}
g_k^TP_k^{-1}g_k.
\]
Hence the uniform estimate becomes
\[
-g_k^Td_k
\ge
(1-\underline\vartheta^{-2})d_k^TP_kd_k.
\]
For clarity, the main algorithm uses fixed \(\vartheta\).
\end{remark}

\section{Algorithm and a smoothing-parameter-uniform Armijo bound}
\label{sec:algorithm}

We use the ordinary Armijo condition
\begin{equation}
\Phi_\tau(z_k+\alpha d_k)
\le
\Phi_\tau(z_k)+c_1\alpha g_k^Td_k,
\qquad
c_1\in(0,1).
\label{eq:armijo}
\end{equation}
The backtracking factor is \(\rho\in(0,1)\).

\begin{theorem}[\(\tau\)-uniform acceptable step interval]
\label{thm:armijo}
For any direction generated by \eqref{eq:direction} with
\(0\le\mu_k\le\mu_k^\star\), every
\begin{equation}
0<\alpha
\le
2(1-c_1)\zeta_\vartheta
\label{eq:acceptable-alpha}
\end{equation}
satisfies the Armijo condition \eqref{eq:armijo}.
\end{theorem}

\begin{proof}
Corollary~\ref{cor:ray-major} and Theorem~\ref{thm:energy-descent} give
\[
\begin{aligned}
\Phi_\tau(z_k+\alpha d_k)
&\le
\Phi_\tau(z_k)+\alpha g_k^Td_k+\frac{\alpha^2}{2}d_k^TP_kd_k\\
&\le
\Phi_\tau(z_k)
+
\alpha g_k^Td_k
\left(
1-\frac{\alpha}{2\zeta_\vartheta}
\right).
\end{aligned}
\]
Since \(g_k^Td_k<0\), Armijo is guaranteed if
\(1-\alpha/(2\zeta_\vartheta)\ge c_1\), which is exactly \eqref{eq:acceptable-alpha}.
\end{proof}

\begin{corollary}[\(\tau\)-uniform backtracking lower bound]
\label{cor:alpha-lower}
If backtracking starts at \(\alpha=1\) and repeatedly multiplies by
\(\rho\in(0,1)\), then
\begin{equation}
\boxed{
\alpha_k
\ge
\underline\alpha
:=
\rho
\min\left\{
1,\,
2(1-c_1)(1-\vartheta^{-2})
\right\}
>0.
}
\label{eq:alpha-lower}
\end{equation}
The lower bound is independent of the smoothing parameter \(\tau\).
\end{corollary}

\subsection*{Fixed-\(\tau\) algorithm}

\noindent
\fbox{\begin{minipage}{0.95\linewidth}
\textbf{Algorithm 1: Adaptive-curvature hyperbolic-majorization preconditioned three-term CG (fixed \(\tau\)).}

\begin{enumerate}[leftmargin=7mm,itemsep=1mm,topsep=1mm]
\item
Choose \(z_0\), \(\tau>0\), \(\delta>0\), \(\vartheta>1\),
\(c_1,\rho\in(0,1)\), tolerance \(\eps>0\), and a curvature-response rule.
\item
For \(k=0,1,\ldots\), compute \(g_k=\nabla\Phi_\tau(z_k)\).
Stop if \(\norm{g_k}\le\eps\).
\item
Construct \(J_k,D_k^{\rm hyp}\), and \(P_k\) by \eqref{eq:Pk}.
\item
If \(k=0\), set \(d_k=-P_k^{-1}g_k\).
Otherwise set \(s_{k-1}=z_k-z_{k-1}\), generate a nonzero curvature response \(b_k\), and compute \(C_k,B_k,\Delta_k,\chi_k,\cD_k\).
If \(s_{k-1}=0\) or \(b_k=0\), restart with \(d_k=-P_k^{-1}g_k\).
\item
Compute
\[
\varpi_k^2=
\frac{1-\chi_k^2}{(\vartheta+\chi_k)^2},
\qquad
\Gamma_0=\frac{\vartheta^2}{\vartheta^2-1},
\qquad
\mu_k^\star
=
\frac{\Gamma_0-1-\varpi_k^2}{\sqrt{\Gamma_0}+1}.
\]
Choose any \(\mu_k\in[0,\mu_k^\star]\); the default enhanced version uses \(\mu_k=\mu_k^\star\).
\item
Compute \(d_k\) from \eqref{eq:direction}.
Starting from \(\alpha=1\), backtrack \(\alpha\leftarrow\rho\alpha\) until \eqref{eq:armijo} holds.
Set \(z_{k+1}=z_k+\alpha_kd_k\).
\end{enumerate}
\end{minipage}}

\section{Fixed-smoothing global convergence and complexity}
\label{sec:convergence}

\begin{assumption}[Compact initial sublevel set]
\label{ass:compact}
For fixed \(\tau>0\),
\[
\cL_\tau(z_0)
=
\{z:\Phi_\tau(z)\le\Phi_\tau(z_0)\}
\]
is compact.
\end{assumption}

\begin{lemma}[Fixed-\(\tau\) spectral upper bound]
\label{lem:spectral}
Under Assumptions~\ref{ass:smooth} and \ref{ass:compact}, there exists
\(M_\tau<\infty\) such that
\[
P_k\preceq M_\tau I
\]
for every accepted iterate.
One possible bound is
\begin{equation}
M_\tau
=
\delta+\bar L
+
\frac1{2\tau}
\sum_{i=1}^m(G_i^2+1),
\label{eq:Mtau}
\end{equation}
where \(G_i\) bounds \(\norm{\nabla f_i(x)}\) on the \(x\)-projection of \(\cL_\tau(z_0)\).
\end{lemma}

\begin{proof}
Armijo and Theorem~\ref{thm:descent} imply monotonic decrease, so all accepted iterates remain in \(\cL_\tau(z_0)\).
The \(x\)-projection of this compact set is compact, and continuity of \(\nabla f_i\) yields finite \(G_i\).
Since \(\omega_{i,k}\ge\tau\),
\[
D_k^{\rm hyp}\preceq\tau^{-1}I.
\]
Therefore
\[
\lambda_{\max}(P_k)
\le
\delta+\bar L
+
\frac1{2\tau}\norm{J_k}^2
\le
\delta+\bar L
+
\frac1{2\tau}\norm{J_k}_F^2,
\]
and
\[
\norm{J_k}_F^2
=
\sum_i(\norm{\nabla f_i(x_k)}^2+1)
\le
\sum_i(G_i^2+1).
\]
\end{proof}

\begin{theorem}[Global first-order convergence for fixed \(\tau\)]
\label{thm:global-conv}
Under Assumptions~\ref{ass:smooth} and \ref{ass:compact}, if Algorithm~1 generates an infinite sequence and at every nonrestart iteration
\[
b_k\neq0,
\qquad
0\le\mu_k\le\mu_k^\star,
\]
then
\begin{equation}
\boxed{
\norm{\nabla\Phi_\tau(z_k)}
\longrightarrow0.
}
\label{eq:grad-zero}
\end{equation}
\end{theorem}

\begin{proof}
From Armijo, Theorem~\ref{thm:descent}, and Corollary~\ref{cor:alpha-lower},
\[
\Phi_\tau(z_{k+1})
\le
\Phi_\tau(z_k)
-
c_1\underline\alpha\,
g_k^TP_k^{-1}g_k.
\]
Because \(\Phi_\tau\) has a finite lower bound on the compact initial sublevel set,
\[
\sum_{k=0}^\infty g_k^TP_k^{-1}g_k<\infty,
\]
hence
\[
g_k^TP_k^{-1}g_k\to0.
\]
Lemma~\ref{lem:spectral} gives
\[
P_k^{-1}\succeq M_\tau^{-1}I,
\]
so
\[
g_k^TP_k^{-1}g_k
\ge
M_\tau^{-1}\norm{g_k}^2.
\]
Therefore \(\norm{g_k}\to0\).
\end{proof}

\begin{corollary}[Fixed-\(\tau\) first-order complexity]
\label{cor:complexity}
Let
\[
\Delta_\tau
=
\Phi_\tau(z_0)
-
\inf_{z\in\cL_\tau(z_0)}\Phi_\tau(z).
\]
Then for every \(N\ge0\),
\begin{equation}
\boxed{
\min_{0\le k\le N}\norm{g_k}^2
\le
\frac{M_\tau\Delta_\tau}
{c_1\underline\alpha(N+1)}.
}
\label{eq:complexity}
\end{equation}
Consequently, \(O(\eps^{-2})\) inner iterations are sufficient to obtain
\(\norm{g_k}\le\eps\).
\end{corollary}

\begin{proof}
Summing the decrease estimate in the proof of Theorem~\ref{thm:global-conv},
\[
\sum_{k=0}^N g_k^TP_k^{-1}g_k
\le
\frac{\Delta_\tau}{c_1\underline\alpha}.
\]
Using
\(g_k^TP_k^{-1}g_k\ge M_\tau^{-1}\norm{g_k}^2\)
gives \eqref{eq:complexity}.
\end{proof}

\begin{remark}
The Armijo lower bound \(\underline\alpha\) is uniform in \(\tau\), but the Euclidean conversion constant \(M_\tau\) in \eqref{eq:Mtau} may grow as \(\tau\downarrow0\).
Thus Corollary~\ref{cor:complexity} is a fixed-\(\tau\) complexity result, not a \(\tau\)-uniform Euclidean gradient complexity statement.
\end{remark}

\section{Continuation and Clarke stationarity of the original minimax problem}
\label{sec:continuation}

Consider smoothing parameters and inner tolerances satisfying
\begin{equation}
\tau_j\downarrow0,
\qquad
\eps_j\downarrow0.
\label{eq:cont-seq}
\end{equation}
Suppose the \(j\)-th smoothed subproblem produces
\(z^j=(x^j,t^j)\) such that
\begin{equation}
\norm{\nabla\Phi_{\tau_j}(z^j)}
\le
\eps_j.
\label{eq:inner-tol}
\end{equation}

\begin{assumption}[Outer accumulation point]
\label{ass:outer}
The outer sequence \(\{z^j\}\) has an accumulation point.
Along a subsequence, still indexed by \(j\),
\[
z^j\to z^\star=(x^\star,t^\star).
\]
\end{assumption}

\begin{theorem}[Clarke stationarity of continuation accumulation points]
\label{thm:clarke}
Under Assumptions~\ref{ass:smooth} and \ref{ass:outer}, if
\eqref{eq:cont-seq}--\eqref{eq:inner-tol} hold, then
\begin{equation}
\boxed{
0\in\partial_C f(x^\star).
}
\label{eq:clarke-stat}
\end{equation}
\end{theorem}

\begin{proof}
Define
\[
\lambda_i^j
=
\frac12
\left(
1+
\frac{f_i(x^j)-t^j}
{\sqrt{(f_i(x^j)-t^j)^2+\tau_j^2}}
\right).
\]
Then \(0<\lambda_i^j<1\).
By the gradient formula \eqref{eq:gradient} and \eqref{eq:inner-tol},
\[
\sum_i\lambda_i^j\nabla f_i(x^j)\to0,
\qquad
\sum_i\lambda_i^j\to1.
\]
Passing to a further subsequence if necessary, compactness of \([0,1]^m\) gives
\(\lambda_i^j\to\lambda_i\in[0,1]\).
By continuity,
\[
\sum_i\lambda_i=1,
\qquad
\sum_i\lambda_i\nabla f_i(x^\star)=0.
\]

It remains to show that positive limiting weights are supported on the active set.
Let
\[
r_i^\star=f_i(x^\star)-t^\star.
\]
If \(r_i^\star<0\), then \(\lambda_i^j\to0\), hence \(\lambda_i=0\).
If some \(r_i^\star>0\), then the corresponding \(\lambda_i^j\to1\).
Because \(\sum_i\lambda_i=1\), at most one residual can be positive, and that index is the unique maximizer of the component values.
If no residual is positive, then \(f_i(x^\star)\le t^\star\) for all \(i\).
Since the limiting weights sum to one, some \(\lambda_i>0\); such an index cannot have negative residual, hence \(r_i^\star=0\).
Therefore \(t^\star=\max_i f_i(x^\star)\), and every positive limiting weight is active.
In both cases,
\[
\lambda_i>0\quad\Longrightarrow\quad i\in I(x^\star).
\]
Thus
\[
0
=
\sum_{i\in I(x^\star)}
\lambda_i\nabla f_i(x^\star),
\qquad
\lambda_i\ge0,
\qquad
\sum_{i\in I(x^\star)}\lambda_i=1.
\]
Equation \eqref{eq:clarke} implies \eqref{eq:clarke-stat}.
\end{proof}

\section{Curvature-response instances}
\label{sec:instances}

The convergence theory only requires a nonzero response \(b_k\); a curvature interpretation is obtained by choosing
\[
b_k\approx\mathcal H_ks_{k-1},
\]
where \(\mathcal H_k\) represents local curvature.
The following instances keep \(s_{k-1}\) fixed and therefore remain in the recurrence \eqref{eq:three-term}.

\subsection{Instance I: standard secant response}

The default choice is
\begin{equation}
\boxed{
b_k=y_{k-1}=g_k-g_{k-1}.
}
\label{eq:instance-secant}
\end{equation}
When \(\Phi_\tau\) is \(C^2\) along the preceding segment,
\[
y_{k-1}
=
\left[
\int_0^1
\nabla^2\Phi_\tau(z_{k-1}+t s_{k-1})\,\dd t
\right]s_{k-1},
\]
so \((s_{k-1},y_{k-1})\) is a displacement--average-curvature pair.
This instance gives Corollary~\ref{cor:DL} and requires no Hessian information.

\subsection{Instance II: current Hessian-vector response}

If \(\Phi_\tau\in C^2\) and Hessian-vector products are available, one may use
\begin{equation}
\boxed{
b_k^{\rm HV}
=
\nabla^2\Phi_\tau(z_k)s_{k-1}.
}
\label{eq:instance-hv}
\end{equation}
The Hessian can be indefinite; the framework does not require
\(s_{k-1}^Tb_k^{\rm HV}>0\).
If \(b_k^{\rm HV}=0\), one may fall back to the secant response or restart.

\subsection{Instance III: Hessian-free finite-difference response}

For a small \(\epsilon_k>0\),
\begin{equation}
\boxed{
b_k^{\rm FD}
=
\frac{
\nabla\Phi_\tau(z_k+\epsilon_ks_{k-1})
-
\nabla\Phi_\tau(z_k)
}{\epsilon_k}.
}
\label{eq:instance-fd}
\end{equation}
When the Hessian is locally continuous,
\(b_k^{\rm FD}\approx\nabla^2\Phi_\tau(z_k)s_{k-1}\).
This instance exchanges an additional gradient evaluation for local curvature information.

\subsection{Instance IV: quasi-Newton curvature response}

Let \(H_k^{\rm QN}\) be a Hessian approximation generated by a limited-memory or structured update.
Use
\begin{equation}
\boxed{
b_k^{\rm QN}
=
H_k^{\rm QN}s_{k-1}.
}
\label{eq:instance-qn}
\end{equation}
The matrix is used only to supply a curvature response; the base direction remains
\(-P_k^{-1}g_k\).
For the convergence theory above, the only algebraic requirement on the selected response is \(b_k^{\rm QN}\neq0\).

\subsection{Instance V: historical curvature predictor}

Let
\[
S_k^-=[s_{k-2},\ldots,s_{k-\ell-1}],
\qquad
Y_k^-=[y_{k-2},\ldots,y_{k-\ell-1}],
\]
using only older pairs.
A regularized predictor coefficient may be defined by
\[
c_k
=
\arg\min_c
\left\{
\norm{P_k^{1/2}(S_k^-c-s_{k-1})}^2
+
\lambda_c\norm c^2
\right\},
\qquad \lambda_c>0,
\]
and the response by
\begin{equation}
\boxed{
b_k^{\rm MS}
=
Y_k^-c_k.
}
\label{eq:instance-ms}
\end{equation}
If the old gradient differences are approximately generated by a slowly varying local curvature operator, then
\(b_k^{\rm MS}\) predicts the response associated with the current displacement.

\begin{remark}[Iteration-dependent response selection]
The rule generating \(b_k\) may change from one iteration to the next.
For each selected nonzero \(b_k\), one simply recomputes
\(B_k,\Delta_k,\chi_k,\cD_k,\varpi_k\), and \(\mu_k^\star\).
The denominator, descent, energy, Armijo, and fixed-\(\tau\) convergence proofs do not require one fixed response generator.
In particular, \(s_{k-1}^Tb_k>0\) is not assumed.
\end{remark}

\begin{remark}[Curvature interpretation versus convergence]
The modeling relation
\[
b_k\approx\mathcal H_ks_{k-1}
\]
is not a convergence assumption.
It is a design principle that gives the otherwise algebraically admissible response \(b_k\) a meaningful curvature interpretation.
\end{remark}

\section{Large-scale implementation}
\label{sec:implementation}

Although \(P_k\) is an \((n+1)\times(n+1)\) matrix, it need not be formed explicitly.
Set
\[
A
=
\delta I_{n+1}+\bar L\Pi_x
=
\diag((\delta+\bar L)I_n,\delta),
\qquad
U_k
=
\frac1{\sqrt2}(D_k^{\rm hyp})^{1/2}J_k.
\]
Then
\begin{equation}
P_k=A+U_k^TU_k,
\qquad
\operatorname{rank}(U_k^TU_k)\le m.
\label{eq:low-rank}
\end{equation}
The Woodbury identity gives
\begin{equation}
\boxed{
P_k^{-1}
=
A^{-1}
-
A^{-1}U_k^T
\left(
I_m+U_kA^{-1}U_k^T
\right)^{-1}
U_kA^{-1}.
}
\label{eq:woodbury}
\end{equation}
Thus when \(m\ll n\) or \(m\) is moderate, both
\(P_k^{-1}g_k\) and \(P_k^{-1}b_k\) can reuse one factorization of an \(m\times m\) SPD system.
When both \(m\) and \(n\) are large,
\[
P_kv
=
Av+\frac12J_k^T\bigl(D_k^{\rm hyp}(J_kv)\bigr)
\]
can be computed matrix-free.
The default secant instance requires component gradients but no component Hessians.

\section{Numerical experiments}
\label{sec:numerics}

\subsection{Experimental design}

\textit{This section is intentionally left incomplete in the present theoretical draft.}
The numerical study should at least compare:
\begin{itemize}
\item the baseline secant method \((b_k=y_{k-1},\mu_k=0)\);
\item the enhanced secant method \((b_k=y_{k-1},\mu_k=\mu_k^\star)\);
\item one or more variable-curvature instances from Section~\ref{sec:instances};
\item representative finite-minimax smoothing CG methods such as Pang--Du--Ju and Guo--Wan;
\item CG-DESCENT applied to the same smoothing model;
\item where computationally reasonable, hyperbolic Newton and hyperbolic modified L-BFGS baselines.
\end{itemize}
Suggested measures include iterations, function/gradient evaluations, accepted step lengths, CPU time, final smoothing stationarity residual, and an outer Clarke-stationarity surrogate.

\subsection{Test problems and parameter settings}

\textit{To be completed.}

\subsection{Results}

\textit{To be completed. No numerical claims are made in this draft.}

\section{Further remarks and a general pair extension}

\begin{remark}[The preconditioner need not be abstracted in the main paper]
The three-term algebra only requires an SPD metric, but the complete global theory uses the specific hyperbolic majorizer \(P_k\).
Keeping this problem-derived metric fixed avoids turning the paper into a generic preconditioned-CG theory and preserves the direct link between smoothing curvature, direction geometry, and Armijo globalization.
\end{remark}

\begin{remark}[Why the main framework fixes \(s_{k-1}\)]
Allowing both history vectors to vary is mathematically possible, but fixing
\(s_{k-1}=\alpha_{k-1}d_{k-1}\)
ensures the recognizable three-term NCG form \eqref{eq:three-term}.
The broader extension is stated next.
\end{remark}

\begin{theorem}[General \((a_k,b_k)\) metric-pair extension]
\label{thm:general-pair}
Let \(a_k\neq0\), \(b_k\neq0\), and define
\[
A_k=a_k^TP_ka_k,
\qquad
B_k=b_k^TP_k^{-1}b_k,
\qquad
\Delta_k=a_k^Tb_k,
\qquad
\chi_k=\frac{\Delta_k}{\sqrt{A_kB_k}},
\]
\[
\cD_k
=
\Delta_k+\vartheta\sqrt{A_kB_k}.
\]
Consider
\begin{equation}
\begin{aligned}
d_k={}&-P_k^{-1}g_k\\
&+
\left[
\frac{g_k^TP_k^{-1}b_k}{\cD_k}
-\mu_k\frac{g_k^Ta_k}{A_k}
\right]a_k
-
\frac{g_k^Ta_k}{\cD_k}P_k^{-1}b_k.
\end{aligned}
\label{eq:general-pair}
\end{equation}
Then:
\begin{enumerate}[label=(\alph*),leftmargin=8mm]
\item
\(\cD_k\ge(\vartheta-1)\sqrt{A_kB_k}>0\);
\item
\[
g_k^Td_k
=
-g_k^TP_k^{-1}g_k
-\mu_k\frac{(g_k^Ta_k)^2}{A_k};
\]
\item with
\[
\widetilde b_k
=
b_k+\vartheta\sqrt{\frac{B_k}{A_k}}\,P_ka_k,
\]
one has
\[
d_k^T\widetilde b_k
=
-\Theta_k g_k^Ta_k,
\]
where
\[
\Theta_k
=
\frac{\widetilde b_k^TP_k^{-1}\widetilde b_k}
{a_k^T\widetilde b_k}
+
\mu_k
\frac{a_k^T\widetilde b_k}{A_k}
>0;
\]
\item defining
\[
\varpi_k^2
=
\frac{1-\chi_k^2}{(\vartheta+\chi_k)^2}
\]
and \(\mu_k^\star\) by \eqref{eq:mu-star}, every
\(0\le\mu_k\le\mu_k^\star\) satisfies
\[
d_k^TP_kd_k
\le
\Gamma_0 g_k^TP_k^{-1}g_k
\]
and hence \eqref{eq:eta};
\item consequently, the same \(\tau\)-uniform Armijo bound and fixed-\(\tau\) convergence proof apply.
\end{enumerate}
\end{theorem}

\begin{proof}
Parts (a)--(c) repeat the metric Cauchy--Schwarz and cancellation arguments of Section~\ref{sec:direction} with \(s_{k-1}\) replaced by \(a_k\).
For part (d), use the whitened unit vectors
\[
u_k=\frac{P_k^{1/2}a_k}{\sqrt{A_k}},
\qquad
v_k=\frac{P_k^{-1/2}b_k}{\sqrt{B_k}}.
\]
The matrix representation is again \eqref{eq:Mmu}, and the proof of Theorem~\ref{thm:mu-star} is unchanged.
Part (e) follows from Theorem~\ref{thm:armijo} and Theorem~\ref{thm:global-conv}.
\end{proof}

\begin{remark}
If \(a_k\not\parallel d_{k-1}\), the direction in Theorem~\ref{thm:general-pair} is better viewed as a metric three-vector extension rather than a classical three-term NCG recurrence.
This is why the main algorithm uses \(a_k=s_{k-1}\) and reserves Theorem~\ref{thm:general-pair} as a structural extension.
\end{remark}

\section{Conclusion}

We developed a hyperbolic-majorization preconditioned three-term NCG framework for nonconvex finite minimax optimization.
The analytic metric \(P_k\) is not an externally chosen preconditioner: it is derived from a global upper model of the hyperbolic smoothing objective and therefore links smoothing curvature, search geometry, and Armijo globalization.
Within this metric, a fixed-displacement/variable-curvature direction admits an adaptive energy-budget enhancement that preserves the baseline worst-case Armijo constant while adding a nonnegative extra descent contribution.
The resulting theory gives regularized conjugacy, sufficient descent, a \(\tau\)-uniform line-search bound, fixed-\(\tau\) first-order convergence and complexity, and continuation to Clarke stationarity.
The numerical evaluation and the comparative behavior of the different curvature-response instances remain to be completed.

\end{document}